\RequirePackage{fix-cm}
\documentclass[smallextended]{svjour3}       
\smartqed  
\usepackage{graphicx}
\usepackage{times,a4wide,mathrsfs}
\usepackage{amsmath}
\usepackage{amsfonts}

\newcommand{\C}{\mathbb{C}}

\newcommand{\QQ}{\mathbb{Q}}
\newcommand{\NN}{\mathbb{N}}
\newcommand{\PP}{\mathbb{P}}

\newcommand{\OO}{\mathcal O}

\newcommand{\grif}{\hbox{Griff}}

\newcommand{\gr}{\hbox{Gr}}
\newcommand{\wt}{\widetilde}
\newcommand{\ima}{\hbox{Im}}
\newcommand{\rom}{\romannumeral}

\newtheorem{convention}{Conventions}

\newtheorem{nonumbering}{Theorem}

\newtheorem{nonumberingcon}{Conjecture}

\begin{document}

\title{On a multiplicative version of Bloch's conjecture}


\author{Robert Laterveer 
}


\institute{CNRS - IRMA, Universit\'e de Strasbourg \at
              7 rue Ren\'e Descartes \\
              67084 Strasbourg cedex\\
              France\\
              \email{laterv@math.unistra.fr}           
           }

\date{Received: date / Accepted: date}

\maketitle

\begin{abstract} A theorem of Esnault, Srinivas and Viehweg asserts that if the Chow group of $0$--cycles of a smooth complete complex variety decomposes, then the top--degree coherent cohomology group decomposes similarly. In this note, we prove that (a weak version of) the converse holds for varieties of dimension at most $5$ that have finite--dimensional motive and satisfy the Lefschetz standard conjecture. The proof is based on Vial's construction of a refined Chow--K\"unneth decomposition for these varieties.
\end{abstract}


\section{Introduction}

Let $X$ be a smooth complete variety of dimension $n$ defined over $\C$.
In the 1992 paper \cite{ESV}, Esnault, Srinivas and Viehweg study the multiplicative behaviour of the Chow ring $A^\ast X$ versus the multiplicative behaviour of the cohomology of $X$. We now state the part of their result that is relevant to us. For a given partition 
   $n=n_1+\cdots+n_r$
   (with $n_i\in\NN_{>0}$),
let us consider the following properties:

\vskip0.5cm
\noindent
(P1) There exists a Zariski open $V\subset X$, such that intersection product induces a surjection
  \[   A^{n_1}V_{\QQ}\otimes A^{n_2}V_{\QQ}\otimes\cdots\otimes A^{n_r}V_{\QQ}\ \to\ A^nV_{\QQ}\ ;\]

  \noindent
  (P2) There exists a Zariski open $V\subset X$, such that cup product induces a surjection
  \[ H^{n_1}(V,\QQ)\otimes H^{n_2}(V,\QQ)\otimes\cdots\otimes H^{n_r}(V,\QQ)\ \to\ H^n(V,\QQ)/N^1\ \]
  (here $N^\ast$ denotes the coniveau filtration);
  
  \noindent
  (P3) Cup product induces a surjection
    \[ H^{n_1}(X,\OO_X)\otimes H^{n_2}(X,\OO_X)\otimes\cdots\otimes H^{n_r}(X,\OO_X)\ \to\ H^n(X,\OO_X)\ .\]

\vskip0.5cm

In these terms, what Esnault, Srinivas and Viehweg prove is the following:

\begin{nonumbering}[Esnault--Srinivas--Viehweg \cite{ESV}] Let $X$ be a smooth complete variety of dimension $n$ over $\C$. Then (P1) implies (P3), and (P2) implies (P3).
\end{nonumbering}

The implication from (P1) to (P3) is a kind of Mumford theorem \cite{M}, and the proof in \cite{ESV} is motivated by Bloch's proof \cite{B}, \cite{B0} of Mumford's theorem using a ``decomposition of the diagonal'' argument. As noted in \cite[remark 2]{ESV}, the generalized Hodge conjecture would imply that (P2) and (P3) are equivalent.\footnote{It is somewhat frustrating that it is not known unconditionally whether (P1) implies (P2), i.e. without assuming the generalized Hodge conjecture. Apparently Esnault, Srinivas and Viehweg had claimed to prove this in an earlier version of their paper, but the argument was found to be incomplete \cite[remark 2]{ESV}.}

It seems natural to conjecture the converse implication (this is discussed in \cite[remark 3]{ESV}):

\begin{nonumberingcon} Let $X$ be a smooth complete variety. Then (P2) implies (P1).
\end{nonumberingcon}

This can be considered a multiplicative version of Bloch's conjecture; indeed for surfaces of geometric genus $0$ this conjecture is equivalent to Bloch's conjecture \cite[remark 3]{ESV}.

The object of this note is to show this conjecture can be proven in some special cases:

\begin{nonumbering} Let $X$ be a smooth projective variety of dimension $n\le 3$, rationally dominated by curves. Then (P2) implies (P1).
\end{nonumbering}

This follows from a more general statement (theorem \ref{main}). Actually, our argument works for varieties $X$ of dimension up to $5$ that satisfy the Lefschetz standard conjecture and have finite--dimensional motive in the sense of Kimura \cite{Kim}, provided we replace (P2) by a variant involving Vial's niveau filtration $\wt{N}^\ast$ \cite{V4} instead of $N^\ast$. 

This result is hardly surprising: since the appearance of Kimura's landmark paper \cite{Kim}, where it is shown that finite--dimensionality implies the Bloch conjecture for surfaces, there have been a great many results attesting to the usefulness of finite--dimensionality \cite{J4}, \cite{An}, \cite{Iv}, \cite{Kah}, \cite{KMP}, \cite{GP}, \cite{GP2}, \cite{V4}, \cite{Y}, \cite{Lat}. The present note is but one more instance of this general principle, illustrating how nicely finite--dimensionality allows to bridge the abyss separating homological equivalence from rational equivalence.

\begin{convention} In this note, the word {\sl variety\/} will refer to an irreducible reduced scheme of finite type over $\C$, endowed with the Zariski topology. A {\sl subvariety\/} is a (possibly reducible) reduced subscheme which is equidimensional. The Chow group of $j$--dimensional algebraic cycles on $X$ with $\QQ$--coefficients modulo rational equivalence is denoted $A_jX$; for $X$ smooth of dimension $n$ the notations $A_jX$ and $A^{n-j}X$ will be used interchangeably. Caveat: note that what we denote $A^jX$ is elsewhere often denoted $CH^j(X)_{\QQ}$.
The Griffiths group $\grif^j$ is the group of codimension $j$ cycles that are homologically trivial modulo algebraic equivalence, again with $\QQ$--coefficients. In an effort to lighten notation, we will often write $H^jX$ or $H_jX$ to indicate singular cohomology $H^j(X,\QQ)$ resp. Borel--Moore homology $H_j(X,\QQ)$.
\end{convention}

\section{Preliminaries}

Let $X$ be a smooth projective variety of dimension $n$, and $h\in H^2(X,\QQ)$ the class of an ample line bundle. The hard Lefschetz theorem asserts that the map
  \[  L^{n-i}\colon H^i(X,\QQ)\to H^{2n-i}(X,\QQ)\]
  obtained by cupping with $h^{n-i}$ is an isomorphism, for any $i< n$. One of the standard conjectures asserts that the inverse isomorphism is algebraic.

\begin{definition}[Lefschetz standard conjecture]{} Given a variety $X$, we say that $B(X)$ holds if for all ample $h$, and all $i<n$ the isomorphism 
  \[  (L^{n-i})^{-1}\colon 
  H^{2n-i}(X,\QQ)\stackrel{\cong}{\rightarrow} H^i(X,\QQ)\]
  is induced by a correspondence.
 \end{definition}  
 
 \begin{remark} It is known that $B(X)$ holds for the following varieties: curves, surfaces, abelian varieties \cite{K0}, \cite{K}, 3folds not of general type \cite{Tan}, $n$--dimensional varieties $X$ which have $A_i(X)_{}$ supported on a subvariety of dimension $i+2$ for all $i\le{n-3\over 2}$ \cite[Theorem 7.1]{V}, $n$--dimensional varieties $X$ which have $H_i(X)=N^{\llcorner {i\over 2}\lrcorner}H_i(X)$ for all $i>n$ \cite[Theorem 4.2]{V2}, products and hyperplane sections of any of these.
 \end{remark}

\begin{definition}[Coniveau filtration \cite{BO}]\label{con} Let $X$ be a quasi--projective variety. The {\em coniveau filtration\/} on cohomology and on homology is defined as
  \[\begin{split}   N^c H^i(X,\QQ)&= \sum \ima\bigl( H^i_Y(X,\QQ)\to H^i(X,\QQ)\bigr)\ ;\\
                           N^c H_i(X,\QQ)&=\sum \ima \bigl( H_i(Z,\QQ)\to H_i(X,\QQ)\bigr)\ ,\\
                           \end{split}\]
   where $Y$ runs over codimension $\ge c$ subvarieties of $X$, and $Z$ over dimension $\le i-c$ subvarieties.
 \end{definition}



  
  

Vial introduced the following variant of the coniveau filtration:

\begin{definition}[Niveau filtration \cite{V4}] Let $X$ be a smooth projective variety. The {\em niveau filtration} on homology is defined as
  \[ \wt{N}^j H_i(X)=\sum_{\Gamma\in A_{i-j}(Z\times X)_{}} \ima\bigl( H_{i-2j}(Z)\to H_i(X)\bigr)\ ,\]
  where the union runs over all smooth projective varieties $Z$ of dimension $i-2j$, and all correspondences $\Gamma\in A_{i-j}(Z\times X)_{}$.
  The niveau filtration on cohomology is defined as
  \[   \wt{N}^c H^iX:=   \wt{N}^{c-i+n} H_{2n-i}X\ .\]
  
\end{definition}

\begin{remark}
The niveau filtration is included in the coniveau filtration:
  \[ \wt{N}^j H^i(X)\subset N^j H^i(X)\ .\] 
  These two filtrations are expected to coincide; indeed, Vial shows this is true if and only if the Lefschetz standard conjecture is true for all varieties \cite[Proposition 1.1]{V4}. 
  \end{remark}

The main ingredient we will use in this note is Kimura's nilpotence theorem:

\begin{theorem}[Kimura \cite{Kim}]\label{nilp} Let $X$ be a smooth projective variety of dimension $n$ with finite--dimensional motive. Let $\Gamma\in A^n(X\times X)_{}$ be a correspondence which is numerically trivial. Then there is $N\in\NN$ such that
     \[ \Gamma^{\circ N}=0\ \ \ \ \in A^n(X\times X)_{}\ .\]
\end{theorem}

 We refer to \cite{Kim}, \cite{An}, \cite{Iv}, \cite{MNP} for the definition of finite--dimensional motive. Conjecturally, any variety has finite--dimensional motive \cite{Kim}. What mainly concerns us in the scope of this note, is that there are quite a few non--trivial examples, giving rise to interesting applications:
 
\begin{remark} 
The following varieties have finite--dimensional motive: abelian varieties, varieties dominated by products of curves \cite{Kim}, $K3$ surfaces with Picard number $19$ or $20$ \cite{P}, surfaces not of general type with vanishing geometric genus \cite[Theorem 2.11]{GP}, Godeaux surfaces \cite{GP}, 3folds and $4$folds with nef tangent bundle \cite{I}, \cite[Example 3.16]{V3}, \cite{I2}, certain 3folds of general type \cite[Section 8]{Vial}, varieties of dimension $\le 3$ rationally dominated by products of curves \cite[Example 3.15]{V3}, varieties $X$ with $A^i_{AJ}X_{\QQ}=0$ for all $i$ \cite[Theorem 4]{V2}, products of varieties with finite--dimensional motive \cite{Kim}.
\end{remark}

\begin{remark}
It is worth pointing out that up till now, all examples of finite-dimensional motives happen to be in the tensor subcategory generated by Chow motives of curves. On the other hand, ``many'' motives
are known to lie outside this subcategory, e.g. the motive of a general hypersurface in $\PP^3$ \cite[Remark 2.34]{Ay} (here ``general'' means ``outside a countable union of 
Zariski--closed proper subsets'').
\end{remark}

There exists another nilpotence result, which predates and prefigures Kimura's theorem:

\begin{theorem}[Voisin \cite{V9}, Voevodsky \cite{Voe}]\label{VV} Let $X$ be a smooth projective algebraic variety of dimension $n$, and $\Gamma\in A^n(X\times X)_{}$ a correspondence which is algebraically trivial. Then there is $N\in\NN$ such that
     \[ \Gamma^{\circ N}=0\ \ \ \ \in A^n(X\times X)_{}\ .\]
    \end{theorem}

\section{Main result}

In this section, we prove the main result of this note:

\begin{theorem}\label{main} Let $X$ be a smooth projective variety over $\C$ of dimension $n$. Suppose

\noindent{(\rom1)} $n\le 5$;

\noindent{(\rom2)} $B(X)$ is true;

\noindent{(\rom3)} $X$ has finite--dimensional motive, or $\grif^n(X\times X)=0$.


Given $n_i\in\NN$ with $n=n_1+\cdots n_r$, suppose that

\noindent{(P2)} Cup product induces a surjection
  \[ H^{n_1}(X)\otimes H^{n_2}(X)\otimes\cdots\otimes H^{n_r}(X)\ \to\ H^n(X)/\wt{N}^1\ .\]

  
Then

\noindent{(P1)}
There exists an open $V\subset X$, such that intersection product induces a surjection
  \[   A^{n_1}V_{}\otimes A^{n_2}V_{}\otimes\cdots\otimes A^{n_r}V_{}\ \to\ A^nV_{}\ .\]
  \end{theorem}

  In dimension $\le 3$, the niveau filtration $\wt{N}^1$ can be replaced by the coniveau filtration $N^1$, and we obtain the result announced in the introduction:
  
  \begin{corollary}\label{cor} Let $X$ be as in theorem \ref{main}, and of dimension $n\le 3$. Then condition (P2) can be replaced by
  
  \noindent{(P2$^\prime$)} There is an open $V\subset X$ such that cup product induces a surjection
  \[ H^{n_1}(V)\otimes\cdots\otimes H^{n_r}(V)\ \to\ H^n(V)/{N}^1\ .\]
  
  If $n=2$, condition (P2) can be replaced by
  
  \noindent{(P2$^{\prime\prime}$)} There is an open $V\subset X$ such that cup product induces a surjection
  \[ H^{1}(V)\otimes H^{1}(V)\ \to\ H^2(V)/{F}^1\ ,\]
  where $F^\ast$ denotes the Hodge filtration.
  \end{corollary}
  
  \begin{proof}(of corollary \ref{cor}) It is easily seen that (P2$^\prime$) implies surjectivity of
  \[ H^{n_1}(X)\otimes\cdots\otimes H^{n_r}(X)\ \to\ H^n(X)/{N}^1\ .\]
 Suppose now $n=3$. Since
    \[ N^1H^3X=\wt{N}^1H^3X\]
    \cite[page 415 "Properties"]{V4}, the above is equivalent to condition (P2) of theorem \ref{main}.
    The case $n=2$ is similar, using that
    \[ \wt{N}^1 H^2X=N^1 H^2X=F^1 H^2X\ .\]
   \end{proof}

\begin{proof}(of theorem \ref{main}) 

Since $X$ satisfies $B(X)$, the K\"unneth components $\pi_i$ are algebraic \cite{K0}, \cite{K}.
Because the dimension is at most $5$, the variety $X$ satisfies conditions (*) and (**) of Vial's \cite{V4}. This implies the existence of a refined Chow--K\"unneth decomposition of the diagonal
  \[ \Delta =\sum_{i,j} \Pi_{i,j}\ \ \in A^n(X\times X)\ \]
  (loc. cit., Theorems 1 and 2).  Here the $\Pi_{i,j}$ are mutually orthogonal idempotents, which act on homology as projectors
  \[  (\Pi_{i,j})_\ast\colon\ \ H_\ast(X)\ \to\ \gr^j_{\wt{N}} H_i(X)\to H_\ast(X)\ .\]
  (Note that $\gr^j_{\wt{N}} H_i(X)$ is a priori not a subspace of $H_i(X)$; however, over $\C$ the existence of a polarization gives a canonical identification with a subspace of $H_i(X)$ \cite{V4}.)
  Assumption (P2) translates into the fact that the morphism of homological motives
    \[ f=\Pi_{n,0}\circ\Delta^r\circ(\pi_{n_1}\times\cdots\times\pi_{n_r})\colon\ \ (X,\pi_{n_1})\otimes\cdots\otimes(X,\pi_{n_r})\ \to\ (X,\Pi_{n,0})\ \ \in {\mathcal M}_{hom}\]
  is surjective, where $\Delta^r$
    is the class in $A^{nr}(X^{r+1})$ of the ``small diagonal'' 
  \[ \Delta^r:=\{(x,x,\ldots,x), x\in X\}\ \ \subset X^{r+1}\ .\]
 
To see this,      
 note that since $B(X)$ holds and we are in characteristic $0$, homological and numerical equivalence coincide on $X$ and its powers \cite{K0}, \cite{K}.
  Thus, using Jannsen's semisimplicity result \cite{J1}, the motives $(X,\pi_i)$ and $(X,\Pi_{n,0})$ are contained in a full semisimple abelian subcategory
  ${\mathcal M}_0$ of the category ${\mathcal M}_{hom}$ of motives with respect to homological equivalence (for ${\mathcal M}_0\subset{\mathcal M}_{hom}$ one can take the subcategory generated by varieties that are known to satisfy the Lefschetz standard conjecture). Hence, we get a decomposition
     \[  (X,\Pi_{n,0})= \ima f\oplus M^\prime\ \ \in {\mathcal M}_0\subset{\mathcal M}_{hom}\ .\]
     But assumption (P2) gives that
     \[ H^\ast(M^\prime)=0\ ,\]
     so that $M^\prime=0\in{\mathcal M}_{hom}$.
     
By semisimplicity of ${\mathcal M}_0$, the surjection
       \[ f=\Pi_{n,0}\circ\Delta^r\circ(\pi_{n_1}\times\cdots\times\pi_{n_r})       \colon\ \ (X,\pi_{n_1})\otimes\cdots\otimes(X,\pi_{n_r})\ \to\ (X,\Pi_{n,0})\ \ \in {\mathcal M}_{hom}\]
 is a split surjection. That is, there exists a correspondence $C\in A^{n}(X\times X^r)$ such that
   \[    \Pi_{n,0}= \Pi_{n,0}\circ\Delta^r\circ (\pi_{n_1}\times\cdots\times\pi_{n_r})\circ C \ \ \in H^{2n}(X\times X)\ .\] 
 For brevity, we will henceforth write 
   \[   \Delta_{(n_1,\ldots,n_r)}:=  \Pi_{n,0}\circ\Delta^r\circ (\pi_{n_1}\times\cdots\times\pi_{n_r})\circ C \ \ \in A^n(X\times X)\ ,\]
 where we suppose we have made the following choice for the K\"unneth components $\pi_{n_i}$ modulo rational equivalence:
   \[  \pi_{n_i}=\sum_j \Pi_{n_i,j}\ \ \in A^n(X\times X)\ .\]  
        
 Since by construction \cite[Theorems 1 and 2]{V4}, all the $\Pi_{i,j}$ for $(i,j)\not=(n,0)$ are supported on $(D\times X)   \cup (X\times D)$, for some divisor $D\subset X$, we get an equality modulo homological equivalence
   \[ \Delta=\Pi_{n,0}+\sum_{(i,j)\not=(n,0)} \Pi_{i,j}= \Delta_{(n_1,\ldots,n_r)}+\Gamma_1+\Gamma_2\ \ \in H^{2n}(X\times X)\ ,\]     
with $\Gamma_1,\Gamma_2$  supported on $D\times X$ (resp. on $X\times D$).  

Using one of the two nilpotence theorems (theorem \ref{nilp} in case $X$ has finite--dimensional motive, theorem \ref{VV} in case the Griffiths group vanishes), it follows there exists $N\in\NN$ such that
  \[   \Bigl(\Delta- \Delta_{(n_1,\ldots,n_r)}-\Gamma_1-\Gamma_2\Bigr)^{\circ N}=0\ \ \in A^{n}(X\times X)\ .\]  
 Developing this expression gives
  \[  \Delta=\sum_k Q_k\ \ \in A^n(X\times X)\ ,\]
  where each $Q_k$ is a composition of $\Delta_{(n_1,\ldots,n_r)}$ and $\Gamma_1$ and $\Gamma_2$:
  \[   Q_k=Q_k^0\circ Q_k^1\circ\cdots\circ Q_k^{N^\prime}\ \ \in A^n(X\times X)\ ,\]
 for $Q_k^i\in\{ \Delta_{(n_1,\ldots,n_r)}, \Gamma_1,\Gamma_2\}$.
       
  For reasons of dimension, $Q_k$ does not act on $A^nX$ as soon as $Q_k$ contains at least one copy of $\Gamma_1$. It follows that
   \begin{equation}\label{one}  A^nX=\Delta_\ast A^nX= (\Delta_{(n_1,\ldots,n_r)}\circ(\hbox{something}))_\ast A^nX+  (\Gamma_2\circ(\hbox{something}))_\ast A^nX\ .\end{equation}
   
  It is convenient to rewrite this as follows: define
    \[  \Delta^\prime_{(n_1,\ldots,n_r)}:=\Delta^r\circ (\pi_{n_1}\times\cdots\times\pi_{n_r})\circ C \ \ \in A^n(X\times X) \ ,\]
    so that 
    \[   \Delta_{(n_1,\ldots,n_r)}=\Pi_{n,0}\circ \Delta^\prime_{(n_1,\ldots,n_r)}\ \ \in A^n(X\times X)\ .\]   
  Since all the $\Pi_{i,j}$ with $(i,j)\not=(n,0)$ are supported on $(D\times X)\cup (X\times D)$ \cite[Theorems 1 and 2]{V4}, we have an equality
   \[   \Delta_{(n_1,\ldots,n_r)}=(\Delta-\sum_{(i,j)\not=(n,0)}\Pi_{i,j})\circ \Delta^\prime_{(n_1,\ldots,n_r)}=\Delta^\prime_{(n_1,\ldots,n_r)}+\Gamma_1+\Gamma_2\ \ \in A^n(X\times X)\ ,\]
   with $\Gamma_1, \Gamma_2$ supported on $D\times X$ resp. on $X\times D$.
  Now, equation \eqref{one} can be rewritten as 
   \[A^nX= (\Delta^\prime_{(n_1,\ldots,n_r)}\circ(\hbox{something}))_\ast A^nX+  (\Gamma_2\circ(\hbox{something}))_\ast A^nX\ . \]

 The first term decomposes (lemma \ref{lemma} below), and the second term is supported on $D$; this proves the theorem with $V=X\setminus D$.
 
 \begin{lemma}\label{lemma} Set--up as above. Then
   \[   (\Delta^\prime_{(n_1,\ldots,n_r)})_\ast A^nX\subset \ima \bigl( A^{n_1}X\otimes \cdots \otimes A^{n_r}X\ \xrightarrow{\iota}\ A^nX\bigr)\ ,\]
   where $\iota$ denotes intersection product.
 \end{lemma}
 
 \begin{proof} Recall that by definition,
   \[ \Delta^\prime_{(n_1,\ldots,n_r)}=  \Delta^r\circ (\pi_{n_1}\times\cdots\times\pi_{n_r})\circ C\ \ \in A^n(X\times X)\ .\] 
    
   The point is that the $\pi_{n_i}$ are supported on $Y_{n_i}\times X$, for some $Y_{n_i}\subset X$ of dimension $n_i$, i.e. there exist correspondences $\pi_{n_i}^\prime\in A_n(Y_{n_i}\times X)$ pushing forward to $\pi_{n_i}$ \cite[Theorems 1 and 2]{V4} (actually, this can even be achieved with $Y_{n_i}$ a smooth hyperplane section of dimension $n_i$ \cite[Theorem 7.7.4]{KMP}). The action of the correspondence $\Delta^\prime_{(n_1,\ldots,n_r)}$
   on $0$--cycles thus factors as
     \[  \begin{array}[c]{cccccccc}    & A^nX&&&      \xrightarrow{(  \Delta^\prime_{(n_1,\ldots,n_r)})_\ast}     &&&A^nX\\
                        &\downarrow{=} &&& &&&\uparrow{=}\\
         &A^nX&\xrightarrow{C_\ast}& A^n(X^r)&\xrightarrow{(\pi_{n_1}\times\cdots\times\pi_{n_r})_\ast}&
      A^n(X^r)& \xrightarrow{(\Delta^r)_\ast}  & A^nX\\
      &&&\downarrow&&\uparrow{=}&&\uparrow{=}\\   
       &&&  A^n(Y_{n_1}\times\cdots\times Y_{n_r}) &\xrightarrow{(\pi_{n_1}^\prime\times\cdots\times\pi_{n_r}^\prime)_\ast}& A^n(X^r)&\xrightarrow{(\Delta^r)_\ast}&A^nX\\
       &&& \uparrow{\times_r}&&\uparrow&&\uparrow{=}\\
      &&& A^{n_1}(Y_{n_1})\otimes\cdots\otimes A^{n_r}(Y_{n_r})&\xrightarrow{(\pi_{n_1}^\prime)_\ast\otimes\cdots\otimes (\pi_{n_r}^\prime)_\ast}& A^{n_1}(X)\otimes\cdots\otimes A^{n_r}(X)   &\xrightarrow{\iota}& A^nX\\
         \end{array}\]
 The arrow labelled $\times_r$ is surjective (because it is a map on $0$--cycles); since the diagram commutes this implies the lemma.  
     \end{proof}
    \end{proof} 
     

 \begin{remark} Here are some non--trivial cases where property (P1) is known to hold: abelian varieties (with $V=X$ and all the $n_j=1$, cf. \cite{Bl}); the variety of lines of a cubic threefold (with $V=X$ and $n_1=n_2=1$, cf. \cite[Example 1.7]{B}); the variety of lines of a cubic fourfold (with $V=X$ and $n_1=n_2=2$, cf. \cite[Theorem 20.2]{SV}). Note that in the last two cases, finite--dimensionality of the motive is not known, so (P1) can {\em not\/} be deduced from theorem \ref{main}; one needs some geometric arguments to establish (P1).
 \end{remark}

  \begin{remark}
  The assumption ``$\grif^n(X\times X)_{}=0$'' in theorem \ref{main} is mainly of theoretical interest, and not practically useful. Indeed, there are precise conjectures (based on the Bloch--Beilinson conjectures) saying how the coniveau filtration on cohomology should influence Griffiths groups \cite{J3}.  
 For $n=2$, it is conjectured that if $H^1(X)=0$ then $\grif^2(X\times X)_{}=0$. For $n=3$, it is conjectured that if $h^{0,2}(X)=h^{0,3}(X)=0$ then $\grif^3(X\times X)_{}=0$. For $n=4$, if $h^{2,0}(X)=h^{3,0}(X)=h^{4,0}(X)=h^{2,1}(X)=0$ then $\grif^4(X\times X)_{}$ should vanish. These predictions are particular instances of \cite[Corollary 6.8]{J3}.
  
  Unfortunately, it seems these conjectures are not known in any non--trivial cases (i.e., outside of the range of varieties with Abel--Jacobi trivial Chow groups); it would be very interesting to find (non--trivial) examples where they can be proven ! 
\end{remark}

  \begin{remark} The Chow motive $(X,\Pi_{n,0})$ (which for varieties verifying (*) and (**) of \cite{V4} is unique up to isomorphism) can be considered the ``most transcendental part'' of the motive $(X,\Delta)$. When $X$ is a surface, $(X,\Pi_{2,0})$ is the transcendental part denoted $t_2(X)$ (and studied in detail) in \cite{KMP}.
  
  Actually, following \cite{KMP} one might hope that $\Pi_{n,0}$ can be linked with the theory of birational motives of Kahn--Sujatha \cite{KS}; this would perhaps give a more conceptual proof (or an extension ?) of theorem \ref{main}. I have not looked into this yet.
  \end{remark}
  
 \begin{remark} The argument of theorem \ref{main} also shows the following: suppose the standard conjecture of Lefschetz type holds universally. Then for any variety with finite--dimensional motive, (P2) implies (P1).
 \end{remark}

\begin{remark} Suppose $X$ satisfies the hypotheses of theorem \ref{main}, so that
  \[  A^{n_1}V_{}\otimes A^{n_2}V_{}\otimes\cdots\otimes A^{n_r}V_{}\ \to\ A^nV\]
  is surjective, for some open $V\subset X$.
  It seems interesting to ask how many simple tensors are needed to generate $A^nV$. For a given $n_1,\ldots,n_r$, let's say $a\in A^nV$ is {\sl $k$--decomposable\/} if there is an expression
    \[ a=\sum_{j=1}^k a_j^1\cdot a_j^2\cdots a_j^r\ \ \in A^nV\ ,\]
    with $a_j^i\in A^{n_i}V$. This is related to unpublished work of Nori, discussed in \cite[remark 5]{ESV}. According to loc. cit., Nori proves that for any $X$ with $H^n(X,\OO_X)\not=0$ and any $k$, there exist elements in $A^nX$ that are not $k$--decomposable (with respect to any $(n_1,\ldots,n_r)$ with $r>1$).\footnote{To be precise, Nori's result is more general, as the notion of $k$--decomposability in \cite{ESV} is broader than the notion discussed here: in loc. cit., an element $a\in A^nV$ is defined to be $k$--decomposable if it can be written as $a=\sum_{j=1}^k a_j\cdot b_j$ with $a_j\cdot b_j$ homogeneous of degree $n$.} It seems likely the same is true for $A^nV$.
    
 The only thing I am able to prove is the following: for $X$ satisfying the hypotheses of theorem \ref{main}, there is an open $V\subset X$ such that each point $v\in V$ is $1$--decomposable in $A^nV$. (To see this, one uses (P1) to obtain a Bloch--Srinivas style decomposition of the diagonal
   \[ \Delta= C_1\cdot C_2\cdots C_r+\Gamma_1+\Gamma_2\ \ \in A^n(X\times X)\ ,\]
   with $C_i\in A^{n_i}(X\times X)$ and $\Gamma_j$ as before; this is done in \cite{ESV} using the method of \cite{BS}. Given a point $v\in V$, let $\tau_v$ denote the inclusion $v\times X\hookrightarrow X\times X$, and let $C_i^v$ denote the restriction $C_i^v=(\tau_v)^\ast(C_i)\in A^{n_i}(v\times X)$.
  Now, note that   
   \[  \begin{split}v=\Delta_\ast v=(C_1\cdots C_r)_\ast v&= (p_2)_\ast\bigl( (v\times X)\cdot C_1\cdots C_r\bigr)\\
                                                                                         &=(p_2)_\ast\bigl( (\tau_v)_\ast (C_1^v\cdots C_r^v)\bigr)\\
                                                                                             & =C_1^v\cdots C_r^v\ \ \in A^nV\ .)
                                                                                             \end{split}\] 
                For general $0$--cycles on the other hand, even when (as in the case of theorem \ref{main}) all $0$--cycles are $k$--decomposable for some $k$, it seems unlikely $k$ can be bounded.                                                                             
 \end{remark}

\begin{acknowledgements}
This note was stimulated by the Strasbourg 2014---2015 ``groupe de travail'' based on the monograph \cite{Vo}. I want to thank all the participants of this groupe de travail for the very pleasant and stimulating atmosphere, and their interesting lectures. Thanks to Charles Vial and the referee for helpful comments. Many thanks to Yasuyo, Kai and Len for providing excellent working conditions at home in Schiltigheim.
\end{acknowledgements}



\end{document}